\documentclass[submission]{FPSAC2023}

\usepackage{bm}

\usepackage[capitalize]{cleveref}

\usepackage{todonotes}

\newcommand{\Q}{\mathbb{Q}}

\newcommand{\val}{\textsf{val}}
\newcommand{\LL}{\mathbb{L}}
\newcommand{\K}{\mathbb{K}}
\DeclareMathOperator{\Jac}{\mathsf{Jac}}
\DeclareMathOperator{\Det}{\mathsf{Det}}
\DeclareMathOperator{\Sdup}{\mathcal{S}_{\mathsf{dup}}}

\usepackage{mathtools}
\usepackage{csquotes}
\usepackage{amsfonts}
\usepackage{array}
\usepackage{comment}
\usepackage{float}
\usepackage{url}
\usepackage{tikz}
\usepackage{soul}
\usepackage{siunitx}
\usepackage{MnSymbol}
\usepackage[utf8]{inputenc}
\usepackage{caption}
\usepackage{fancyvrb}
\usepackage[linesnumbered,ruled,vlined]{algorithm2e}
\SetKwInput{KwInput}{Input}                
\SetKwInput{KwOutput}{Output}              
\SetKwComment{Comment}{/* }{ */}

\newtheorem{thm}{Theorem}
\newtheorem{lem}[thm]{Lemma}
\newtheorem{prop}[thm]{Proposition}

\newtheoremstyle{note}
{3pt}
{3pt}
{}
{}
{\bfseries}
{.}
{.5em}
{}
\theoremstyle{note}

\newtheorem{ex}{Example}
\newtheorem*{ex1}{Example 1 (cont.)}
\newtheorem*{ex2}{Example 2 (cont.)}

\usepackage{lipsum}

\title[Systems of DDEs with one catalytic variable]{Effective algebraicity for solutions of systems of functional equations with one catalytic variable}

\author[]{Hadrien Notarantonio\addressmark{1, 2}, Sergey Yurkevich\addressmark{1,3}}

\address{
\addressmark{1}Inria Saclay, France\\ 
\addressmark{2}Sorbonne Université, France\\
\addressmark{3}University of Vienna, Austria}
\received{\today}

\abstract{We study systems of $n \geq 1$ discrete differential equations of order $k\geq1$ in one catalytic variable and provide a constructive and elementary proof of algebraicity of their solutions. This yields effective bounds and a systematic method for computing the minimal polynomials. Our approach is a generalization of the pioneering work by~Bousquet-Mélou and Jehanne (2006).
}

\keywords{algebraic functional equations, combinatorics, algebraic algorithms, algebraicity bounds, catalytic variables}

\usepackage[backend=bibtex]{biblatex}
\begin{document}

\maketitle

\section{Introduction}

Numerous combinatorial enumeration problems reduce to the study of functional equations which can be solved by a uniform method introduced by Bousquet-Mélou and Jehanne in the seminal work~\cite{BMJ06}. These functional equations, usually called discrete differential equations~(DDEs) with one catalytic variable, involve a bivariate generating function~$F\in\mathbb{Q}[u][[t]]$ associated to the enumeration problem, and are of the form
\begin{equation}\label{eq:DDE}
F(t, u) = f(u) + t \cdot Q(F(t, u), \Delta_a F(t, u), \dots, \Delta_a^k F(t, u), t, u),
\end{equation}
where $k \in \mathbb{N}$ (called the order of the DDE), $f$ and $Q$ are polynomials, and (for some $a \in \Q$, usually $0$ or $1$) $\Delta_a^\ell$ is the $\ell$th iteration of the operator $\Delta_a: \Q[u][[t]] \to \Q[u][[t]]$ defined~by 
\[
\Delta_a F(t,u) := \frac{F(t,u) - F(t,a)}{u-a}.
\]
In their paper, Bousquet-Mélou and Jehanne designed a ``non-linear kernel method'' which allows one to prove that the unique solution of (\ref{eq:DDE}) is always an algebraic function over $\Q(t,u)$. Significantly in practice, this approach yields an algorithm for finding the minimal polynomial of the specialization $F(t,a)$ and of the bivariate series~$F(t,u)$.

The main contribution of the present paper is a generalization of this method to the case of \emph{systems of discrete differential equations}. More precisely, we shall prove the following theorem. Here and in the following, we let $\K$ be a field of characteristic $0$.
\begin{thm}\label{thm:main_thm}
Let $n, k\geq 1$ be integers and $f_1, \ldots, f_n\in\K[u]$, $Q_1, \ldots, Q_n\in\K[y_1, \ldots, y_{n(k+1)}, t, u]$ be polynomials. Set  $\nabla^k F := (F, \Delta_a F, \ldots, \Delta_a^k F)$. Then the system of equations
\begin{equation}\label{eqn:init_system}
    \begin{cases}  
\textbf{\textup{\text{(E}}}_{\textbf{\textup{\text{F}}}_\textbf{\textup{\text{1}}}} \textbf{\textup{\text{):}}}
    \;\;\;F_1 = f_1(u) + t\cdot Q_1(\nabla^k F_1, \ldots, \nabla^k F_n, t, u),\\
   \indent \vdots \hfill \vdots \\
\textbf{\textup{\text{(E}}}_{\textbf{\textup{\text{F}}}_\textbf{{\text{n}}}} \textbf{\textup{\text{):}}}\;\;F_n = f_n(u) + t\cdot Q_n(\nabla^k F_1, \ldots, \nabla^k F_n, t, u)
    \end{cases}
\end{equation}
admits a unique vector of solutions~$(F_1, \ldots, F_n)\in\K[u][[t]]^n$, and all its components are algebraic functions over~$\K(t, u)$.
\end{thm}
The key idea, analogous to the one in~\cite{BMJ06}, for proving this theorem is to define a deformation of (\ref{eqn:init_system}) that ensures the applicability of a multi-dimensional analog of the ``non-linear kernel method''. Stated explicitly, we show in \cref{lem:det_sols} that after deforming the equations as in (\ref{eqn:deformed_system}), the polynomial in $u$ defined by
the determinant of the Jacobian matrix associated to the numerator equations in~\eqref{eqn:init_system} (considered with respect to the $F_i$)
has exactly $nk$ solutions in an extension of the ring $\bigcup_{d \geq 1}\overline{\K}[[t^{1/d}]]$. Then, after a process of ``duplication of variables'', we construct a zero-dimensional and radical polynomial ideal, a non-trivial element of which must be the desired annihilating polynomial. The most difficult step consists in proving the invertibility of a certain Jacobian matrix (\cref{lem:det_Jaci} and \cref{lem:Lambda}) in order to justify the zero-dimensionality. We remark that an alternative, and possibly more practical, strategy is to reduce the initial system to a single functional equation. Our \cref{prop:preservation_Ui} ensures that such a reduction preserves the roots guaranteed in the deformation step, however, as we will show in~\cref{sec:complexity}, this method is not guaranteed to produce a zero-dimensional polynomial ideal in the end.

Similarly to the work by Bousquet-Mélou and Jehanne, our proof is effective, in the sense that it produces an algorithm for finding the minimal polynomials of the power series of interest. Moreover, we can deduce a bound on the algebraicity degree of each $F_i$. When the field~$\mathbb{K}$ is effective (e.g.~$\mathbb{K} = \mathbb{Q}$), we can also bound the arithmetic complexity of our algorithm, that is the number of operations~$(+, -, \times, \div)$ performed in~$\mathbb{K}$. Denoting by $\textsf{totdeg}(P)$ the total degree of a multivariate polynomial $P$, we obtain the following:
\begin{thm}\label{thm:quantitative_estimates}
In the setting of \cref{thm:main_thm}, let $(F_1, \ldots, F_n)\in\K[u][[t]]^n$  be the vector of solutions and $\delta := \max(\deg(f_1),\dots,\deg(f_n), \textsf{totdeg}(Q_1),\dots,\textsf{totdeg}(Q_n))$. Then the algebraicity degree of each $F_i(t, u)$ over~$\mathbb{K}(t, u)$ is bounded by $(2nk\delta)^{2n^3k^2+2n}/(nk)!^{nk}$. If~$\mathbb{K}$ is effective,
    there exists an algorithm computing the minimal polynomial of any $F_i(t, a)$ in $(2nk\delta)^{O(n^2k)}$
    operations in~$\K$.
\end{thm}

Discrete differential equations are ubiquitous in enumerative combinatorics~\cite{BrTu64, Brown65, Tutte62}. Systems of DDEs also appear in a variety of different contexts throughout combinatorics, for instance for  hard particles on planar maps \cite[\S 5.4]{BMJ06}, inhomogeneous lattice paths~\cite{BK20}, or certain orientations with $n$ edges \cite[\S 5]{BoBMDoPe17}. The usual strategy for solving these systems of equations is to try to reduce a given system to a scalar equation and then apply the method of Bousquet-Mélou and Jehanne. This approach is usually ad-hoc and needs to exploit additional structure of the system. Moreover, since the reduced equation is in general not of the form~(\ref{eq:DDE}) anymore, the theory of~\cite{BMJ06} is not guaranteed to work. 

In the literature there exist two methods to overcome these theoretical issues. First, a deep theorem in commutative algebra by Popescu~\cite{Popescu85a}, so-called ``nested Artin approximation'', guarantees that equations of the form~(\ref{eqn:init_system}) always admit an algebraic solution (see also~\cite[Thm. 16]{BoBMDoPe17} for a statement of this theorem). Note that the nested condition is automatically satisfied in this case and that the uniqueness of the solution is obvious. A drawback of using Popescu's theorem, however, is that its proof is highly non-constructive and can only be applied as a ``black box'', whereas in practice one is often interested in the explicit minimal polynomials annihilating the solutions. Secondly, the frequent case when all polynomials $Q_1,\dots,Q_n$ in~(\ref{eqn:init_system}) are linear functions was effectively solved in the recent FPSAC article~\cite{BK20} by Buchacher and Kauers and independently by Asinowski, Bacher, Banderier and Gittenberger~\cite{ABBG20}. From this viewpoint, it is safe to say that our contribution is a common generalization of central results by Bousquet-Mélou and Jehanne~\cite{BMJ06} and the ``vectorial kernel method'' of~\cite{ABBG20, BK20}, and at the same time an effective and elementary proof of a special case of Popescu's theorem~\cite{Popescu85a}.

We shall highlight the following two examples of systems of DDEs in more detail. 

\begin{ex} \label{ex:eq27}
The following system of DDEs for the generating function of certain planar orientations was considered in~\cite[Eq.$(27)$]{BoBMDoPe17} and solved in the same work: 
\begin{equation}
    \begin{cases}\label{eq:eq27}
    \textbf{(E$_{\textbf{F}_1}$):} \;\;
    F_1(t, u) = 1 + t\cdot \big(u +2uF_1(t, u)^2 + 2uF_2(t, 1) + u\frac{F_1(t, u)-uF_1(t, 1)}{u - 1}\big),\\
    \textbf{(E$_{\textbf{F}_2}$):} \;\;
    F_2(t, u) = t\cdot \big(2uF_1(t, u)F_2(t, u) + uF_1(t, u) + uF_2(t, 1) + u\frac{F_2(t, u)-uF_2(t, 1) }{u - 1}\big).
    \end{cases}
\end{equation}
From our perspective,~\eqref{eq:eq27} has the advantage that it does not require any deformation and, as we will show in~\cref{sec:generic_equations_of_first_order}, it can be solved fast by a direct application of our method. It is thus a good illustration of the simplest non-trivial case of our approach.
\end{ex}

\begin{ex}\label{ex:hard_particles} This example of a system of DDEs modelling a particular case of hard particles on planar maps was introduced and solved in~\cite[Section~$11$]{BMJ06}: 
\begin{equation} \label{eq:exhardpart}
    \begin{cases}
    \textbf{(E$_{\textbf{F}_1}$):} \;\;F_1(t, u) = F_2(t, u) + tu^2F_1(t, u)^2 + tu\frac{uF_1(t, u) - F_1(t, 1)}{u-1},\\
    \textbf{(E$_{\textbf{F}_2}$):} \;\;F_2(t, u) = 1 + tsuF_1(t, u)F_2(t, u) + tsu\frac{F_2(t, u) - F_2(t, 1)}{u-1}.
    \end{cases}
\end{equation}
To apply our method directly, a deformation step~(\ref{eqn:deformed_system}) is necessary (see \cref{sec:complexity}).
\end{ex}
\medskip
The structure of the paper is as follows: In \cref{sec:generic_equations_of_first_order} we explain our method in the case of two equations of order one under the genericity assumption that no deformation is necessary. We summarize the method in an algorithm and showcase it explicitly on~\cref{ex:eq27}. \cref{sec:proof_thm} is devoted to the proofs of~\cref{thm:main_thm} and~\cref{thm:quantitative_estimates}. 
In the last~\cref{sec:complexity} we briefly explore an improvement to our approach which in theory has a better algorithmic complexity but which requires a new genericity assumption. We also discuss possible future works.

\section{The case of two generic equations of first order} \label{sec:generic_equations_of_first_order}

Before proving our main theorem in~\cref{sec:proof_thm}, we introduce our method in the situation of two equations
of order~$1$ and under a genericity assumption on the input system.

Starting with~\eqref{eqn:init_system}, we first multiply \textbf{(E$_{\textbf{F}_1}$)} and \textbf{(E$_{\textbf{F}_2}$)} by $(u-a)^{m_1}$ and $(u-a)^{m_2}$ respectively (for $m_1, m_2\in~\mathbb{N}$) in order to obtain a system with polynomial coefficients in~$u$. By a slight abuse of notation, we shall still write \textbf{(E$_{\textbf{F}_1}$)} and \textbf{(E$_{\textbf{F}_2}$)} for those equations. Note that this system induces polynomials $E_1, E_2$ in~$\K[x_1, x_2, z_0, z_1, t, u]$ whose specializations
to $x_1 = F_1(t, u), x_2 = F_2(t, u), z_0 = F_1(t, a), z_1 = F_2(t, a)$ are zero.

\begin{ex1}\label{ex:cont1_toy}Multiplying \textbf{(E$_{\textbf{F}_1}$)} and \textbf{(E$_{\textbf{F}_2}$)} in~\cref{ex:eq27}
by~$u-1$ gives 
\begin{equation*}
    \begin{cases}
    E_1 =  (1 - x_1)\cdot(u-1)+ t\cdot (2u^2x_1^2 - u^2z_0 + 2u^2z_1 - 2ux_1^2 + u^2 + ux_1 - 2uz_1 - u),\\
    E_2 =  x_2\cdot(1-u) + t\cdot (2u^2x_1x_2 + u^2x_1 - 2ux_1x_2 - ux_1 + ux_2 - uz_1).
    \end{cases}
\end{equation*}
\end{ex1}
\noindent
In the spirit of~\cite{BMJ06}, we now take the derivative of both equations with respect to~$u$:
\begin{equation}\label{eqn:starting_point}
    \begin{pmatrix}
    \partial_{x_1}E_1 & \partial_{x_2}E_1\\
    \partial_{x_1}E_2 & \partial_{x_2}E_2
    \end{pmatrix}
    \cdot \begin{pmatrix}
    \partial_u F_1 \\
    \partial_u F_2
    \end{pmatrix} 
     +\begin{pmatrix}
    \partial_u E_1\\
    \partial_u E_2
    \end{pmatrix} = 0.
\end{equation}
Define $\Det:=\partial_{x_1}E_1 \cdot \partial_{x_2}E_2 - \partial_{x_1}E_2 \cdot \partial_{x_2}E_1 \in \K[x_1,x_2,z_0,z_1,t,u]$ to be the determinant of the square matrix above.
One can show that 
$\Det(F_1(t, u), F_2(t, u), F_1(t, a), F_2(t, a), t, u)\in\K[[t]][[u]]$ admits either~$0, 1$ or $2$ distinct non-zero solutions $u=U(t)\in\bigcup_{d\geq 1}\overline{\K}[[t^{1/d}]] =: \overline{\K}[[t^{\frac{1}{\star}}]]$. We assume now that there exist $2$ such solutions $U_1, U_2 \in \overline{\K}[[t^{\frac{1}{\star}}]]$; we prove in \cref{sec:proof_thm} that it is always the case up to the deformation (\ref{eqn:deformed_system}).

Exploiting the idea of~\cite{BK20}, we now define~$v := (\partial_{x_1}E_2, \;\; - \partial_{x_1}E_1) \in \K[x_1,x_2,z_0,z_1,t,u]^2$ and plug $U_1$ for $u$ into~$v$ and~\eqref{eqn:starting_point}. Note that~$v$ is an element of the left-kernel of the square matrix in~\eqref{eqn:starting_point} $\bmod$ $\Det(x_1, x_2, z_0, z_1, t, u)$. After multiplication of \eqref{eqn:starting_point} by~$v$ on the left, we find a new polynomial relation between $F_1(t, U_i), F_2(t, U_i), F_1(t, a), F_2(t, a), t$ and $U_i$, namely $ \partial_{x_1} E_1 \cdot\partial_{u}E_2 - \partial_{x_1}E_2\cdot \partial_{u}E_1 = 0$ when evaluated at $x_1 = F_1(t, U_i), x_2 = F_2(t, U_i), z_0=F_1(t, a), z_1 = F_2(t, a), u = U_i$. We denote this polynomial by~$P \in \K[x_1, x_2, z_0, z_1, t, u]$.
Define the polynomial
system~$\mathcal{S} := (E_1, E_2, \Det, P)\in\K[t][x_1, x_2, z_0, z_1, u]^4$. It admits the
non-trivial solutions $(F_1(t, U_i), F_2(t, U_i), F_1(t, a), F_2(t, a), U_i)\in\overline{\mathbb{K}}[[t^{\frac{1}{\star}}]]^5$, for $i \in\{1, 2\}$.

\begin{ex1}Continuing \cref{ex:eq27}, we find 
\begin{equation*}
    \begin{cases}
    \Det = (4tu^2x_1 - 4tux_1 + tu - u + 1)(2tu^2x_1 - 2tux_1 + tu - u + 1),\\
    P = -2tx_1x_2 - tx_1 + tx_2 - tz_1 - x_2 + P_1\cdot u
     + P_2\cdot u^2 + P_3\cdot u^3, 
    \end{cases}
\end{equation*}
where $P_1, P_2, P_3$ are explicit (but relatively big) polynomials in~$\mathbb{Q}[x_1, x_2, z_0, z_1, t]$.
\end{ex1}

Now, generalizing naturally the steps of~\cite{BMJ06}, we define for $i \in \{0,1\}$ the polynomial systems $\mathcal{S}_i := \mathcal{S}(x_{2i+1},\; x_{2i+2},  \;z_0, \;z_1, \;t, \;u_{i+1})$ by ``duplicating'' variables. We call the situation ``\emph{generic}'' if the ideal $\langle \mathcal{S}_0, \mathcal{S}_1, m\cdot(u_1 - u_2) - 1\rangle $ has dimension~$0$ over~$\K(t)$. In this case, in order to find an annihilating polynomial of $F_1(t,a)$, it is enough to compute a non-zero element of $\langle \mathcal{S}_0, \mathcal{S}_1, m\cdot(u_1 - u_2) -1\rangle  \cap \K[z_0, t]$.
\begin{ex1}Continuing Example~\ref{ex:eq27}, we compute\footnote{All computations in this paper have been performed in Maple using msolve~\cite{msolve}.} a generator of the polynomial ideal
  $\langle \mathcal{S}_0, \mathcal{S}_1, m\cdot(u_1 - u_2) -1\rangle \cap \Q[z_0, t]$.
  It has degree $13$ in $z_0$ and~$14$ in $t$. In particular, it contains in its factors
  the minimal polynomial of $F_1(t, 1)$ given by 
  $64t^3z_0^3 + (48t^3 - 72t^2 + 2t)z_0^2 - (15t^3 - 9t^2 - 19t + 1)z_0 + t^3 + 27t^2 - 19t + 1$.
\end{ex1}
We summarize the presented algorithm in a more compact form:\vspace{-0.35cm}

\begin{algorithm}[!ht]
\DontPrintSemicolon
  \KwInput{A ``generic'' system of two DDEs \textbf{(E$_{\textbf{F}_1}$)}, \textbf{(E$_{\textbf{F}_2}$)} of order~$1$.}
  \KwOutput{A non-zero $R\in\K[z_0, t]$ annihilating $F_1(t, a)$.}
  Replace \textbf{(E$_{\textbf{F}_1}$)} and \textbf{(E$_{\textbf{F}_2}$)} by their respective numerators and denote by $E_1$ and $E_2$ the associated polynomials in~$\K[x_1, x_2, z_0, z_1, t, u]$.\;
  Compute $\Det:=\partial_{x_1}E_1 \cdot \partial_{x_2}E_2 - \partial_{x_1}E_2 \cdot \partial_{x_2}E_1$ and $P := \partial_{x_1}E_1 \cdot \partial_u E_2 - \partial_{x_1}E_2\cdot \partial_u E_1 $.\;
  Set $\mathcal{S} := (E_1, E_2, \Det, P)\subset\K[x_1, x_2, z_0, z_1, t, u]$.\;
  For $0\leq i \leq 1$, define $\mathcal{S}_i := \mathcal{S}(x_{2i+1},\; x_{2i+2},  \;z_0, \;z_1, \;t, \;u_{i+1})$.\;
  \textbf{Return} a non-zero element of 
  $\langle \mathcal{S}_0, \mathcal{S}_1, m\cdot(u_1 - u_2) -1\rangle  
  \cap \K[z_0, t]$.
\caption{\label{algo:algo1}   \vspace{-0.5cm} Solving systems of two discrete differential equations of order~$1$.}
\end{algorithm}\vspace{-0.35cm}

As already stated, \cref{sec:proof_thm} ensures that for a non-generic input, after a deformation, we can use an algorithm in spirit of Algorithm~\ref{algo:algo1}. 

We remark that if the strategy above is applied in the case of a single equation of first order
$F_1 = f(u) + t\cdot Q_1(F_1, \Delta_a F_1,  t, u)$, the presented method simplifies to the  classical algorithm in~\cite{BMJ06} relying on studying the ideal $\langle E_1, \partial_{x_1} E_1, \partial_u E_1 \rangle$. Stated explicitly, $\partial_{x_1} E_1$ plays the role of $\Det$ and $\partial_uE_1$ plays the role of $P$ 
(as we can take here $v = 1$). 

\section{Proofs of \cref{thm:main_thm} and \cref{thm:quantitative_estimates}} \label{sec:proof_thm}
We start by proving \cref{thm:main_thm}. As explained before, the statement and proof can be seen as a generalization of~\cite[Theorem 3]{BMJ06} and \cite[Theorem 2]{BK20}, so several steps are done analogously. Without loss of generality we assume that $a=0$ and set $\Delta := \Delta_0$.

Denote by $m_1, \ldots, m_n$ the least positive integers greater than or equal to~$k$ such that multiplying~\textbf{($\textbf{E}_{\textbf{F}_{\textit{\textbf{i}}}}$)} in \eqref{eqn:init_system} by $u^{m_i}$ gives a polynomial equation in $u$. Set $\beta := \lfloor 2M/k \rfloor $ and $\alpha:= n^2k\cdot (\beta + 1) + nM$, where $M := m_1 + \cdots + m_n$. 
Let $\epsilon$ be a new variable, $\LL := \K(\epsilon)$, and let $(\gamma_{i, j})_{1\leq i, j\leq n}$ be defined by $\gamma_{i, i} = i^k$ and $\gamma_{i, j} = t^\beta$ for $i\neq j$. Then consider the following system which is a deformation of~\eqref{eqn:init_system}:
\begin{equation}\label{eqn:deformed_system}
    \begin{cases}
\textbf{\textup{\text{(E}}}_{\textbf{\textup{\text{G}}}_\textbf{\textup{\text{1}}}} \textbf{\textup{\text{):}}} \;\;\;G_1 = f_1(u) + t^{\alpha}\cdot Q_1(\nabla^kG_1, \nabla^kG_2, \ldots, \nabla^kG_n, t^{\alpha}, u)
    + t\cdot \epsilon^{k}\cdot \sum_{i = 1}^n \gamma_{1, i} \cdot \Delta^kG_i,\\
     \qquad \qquad \;\;\; \vdots \hfill \vdots \\
\textbf{\textup{\text{(E}}}_{\textbf{\textup{\text{G}}}_\textbf{{\text{\textit{n}}}}} \textbf{\textup{\text{):}}} \;\;\;G_n = f_n(u) + t^{\alpha}\cdot Q_n(\nabla^kG_1, \nabla^kG_2, \ldots, \nabla^kG_n, t^{\alpha}, u)
    + t\cdot \epsilon^{k}\cdot \sum_{i = 1}^n \gamma_{n, i}\cdot\Delta^kG_i.
    \end{cases}
\end{equation}
The fixed point nature of these equations still implies that there exists a unique solution $(G_1, \ldots, G_n)\in\LL[u][[t]]^n$.
Remark that the equalities $F_i(t^{\alpha}, u) = G_i(t, u, 0)$ relate the formal power series
solutions of~\eqref{eqn:init_system} and of~\eqref{eqn:deformed_system}.
Hence, showing that each $G_i$ is algebraic over $\LL(t, u)$ is enough to prove \cref{thm:main_thm}. Moreover, as we will see later, the algebraicity of each $G_i$ follows from the algebraicity 
of~$G_1(0), \ldots, \partial_u^{k-1}G_1(0), \ldots, G_n(0), \ldots, \partial_u^{k-1}G_n(0)$. Here, and in what follows, we shall use the short notation $G_i(u)\equiv G_i(t, u, \epsilon)$, $\partial_0G_i(u)\equiv G_i(u), G_i(0), \partial_u G_i(0), \ldots, \partial_u^{k-1}G_i(0)$ and $A(u)\equiv A(\partial_0G_1, \ldots, \partial_0G_n, t, u)$ for any polynomial $A\in\LL[X_1, \ldots, X_n, t, u]$ with $X_j := x_j, z_{k(j-1)}, \ldots, z_{kj-1}$. In the case~$n=1$, this notation implies that for any~$0\leq i \leq k-1$, the variable $z_i$ stands for~$\partial_u^iF_1(t, a)$.

Let us define $Y_{i, 0} := x_i$ and $Y_{i, j} := ({x_i - z_{k(i-1)} - \cdots - \frac{u^{j-1}}{(j-1)!}z_{k(i-1)+j-1})/u^j}$ for $j\geq 1$.
With these definitions, (\ref{eqn:deformed_system}) is equivalent to the following system of polynomial equations
\begin{equation}\label{eq:E_i's}
    \begin{cases}
    E_1 := u^{m_1}\cdot (f_1(u) - x_1 + t^{\alpha}\cdot Q_1(Y_{1, 0}, \ldots, Y_{1, k},Y_{2,0}, \ldots, Y_{n, k}, t^{\alpha}, u) 
    + t\cdot \epsilon^{k}\cdot\sum_{i = 1}^n \gamma_{1, i}\cdot Y_{i, k}) = 0 ,\\
    \indent \vdots \hfill \vdots \\
    E_n :=  u^{m_n}\cdot (f_n(u) - x_n + t^{\alpha}\cdot Q_n(Y_{1, 0}, \ldots, Y_{1, k},Y_{2,0}, \ldots, Y_{n, k}, t^{\alpha}, u)
    + t\cdot \epsilon^{k}\cdot\sum_{i = 1}^n \gamma_{n, i}\cdot Y_{i, k}) = 0.
    \end{cases}
\end{equation}
Like in (\ref{eqn:starting_point}), we take the derivative with respect to~$u$ of these equations and find
\begin{equation}\label{eqn:En_diff_u}
    \begin{pmatrix}
    \partial_{x_1}E_1 & \dots & \partial_{x_n}E_1\\
    \vdots & \ddots & \vdots \\ 
    \partial_{x_1}E_n & \dots  & \partial_{x_n}E_n
    \end{pmatrix}
    \cdot \begin{pmatrix}
    \partial_u G_1 \\
    \vdots \\
    \partial_u G_n
    \end{pmatrix} 
    + \begin{pmatrix}
    \partial_u E_1\\
    \vdots\\
    \partial_u E_n  
    \end{pmatrix} =0. 
\end{equation}
Let $\Det \in \LL[X_1,\dots,X_n,t][u]$ be the determinant of the square matrix $(\partial_{x_j}E_i)_{1\leq i,j \leq n}$ above. The following lemma on the number of distinct solutions to $\Det(u) = 0$ is the first main step in our proof. 
\begin{lem}\label{lem:det_sols}
    $\Det(u) = 0$ admits exactly $nk$ distinct non-zero solutions
    $U_1, \ldots, U_{nk}\in \overline{\LL}[[t^{\frac{1}{\star}}]]$. 
\end{lem}
\begin{proof} Note that we have
\begin{equation*}
   \Det(u) = \det\begin{pmatrix}
    -u^{m_1} + t\epsilon^{k} \gamma_{1,1} u^{m_1-k} &  \cdots & t\epsilon^{k}\gamma_{1, n}u^{m_1-k}\\
    \vdots & \ddots & \vdots \\
    t\epsilon^{k}\gamma_{n, 1}u^{m_n-k}  &\cdots & -u^{m_n} + t\epsilon^{k}\gamma_{n, n}u^{m_n-k}
    \end{pmatrix}
     + O(t^\alpha u^{M-nk}).
\end{equation*}
For every $i$ we first divide the $i^{\text{th}}$ row by~$u^{m_i-k}$. 
Then, using the definition of $\gamma_{i,j}$ and $\alpha,\beta \geq n$, we see that the matrix above becomes diagonal mod $t^{n+1}$ and its determinant mod $t^{n+1}$ simplifies to $\prod_{j=1}^n (-u^k + t \epsilon^k j^k) \bmod t^{n+1}$. Hence, computing the first terms of a solution in~$u$ by using Newton polygons, we find $nk$ distinct solutions $u=U_1(t), \ldots, U_{nk}(t)$ whose first terms are given by
$\zeta^\ell \cdot t^{\frac{1}{k}}\cdot \epsilon + O(t^{\frac{2}{k}}), \ldots, \zeta^\ell \cdot n\cdot t^{\frac{1}{k}}\cdot \epsilon + O(t^{\frac{2}{k}})
\in\overline{\LL}[[t^{\frac{1}{\star}}]]$, for $\zeta$
a $k$-primitive root of unity and for all $1\leq \ell\leq k$. Finally, note that the constant coefficient in $t$ of $\prod_{j=1}^n (-u^k + t \epsilon^k j^k)$ has degree $nk$ so by \cite[Theorem 2]{BMJ06} there cannot be more than $nk$ solutions to $\Det(u)=0$ in $\overline{\LL}[[t^{\frac{1}{\star}}]]$.
\end{proof} 

Now, let $P$ be the determinant of the square matrix $(\partial_{x_j}E_i)_{1\leq i,j \leq n}$ where the last column $(\partial_{x_n}E_1,\dots,\partial_{x_n}E_n)$ is replaced by $(\partial_{u}E_1,\dots,\partial_{u}E_n)$.
It is easy to see with standard linear algebra arguments that if $\Det=0$ then (\ref{eqn:En_diff_u}) implies that $P=0$. Hence, we define the polynomial system $\mathcal{S} := (E_1, \ldots, E_n, \Det, P)$ in~$\LL[t][X_1, \ldots, X_n, u]$. We see that $\mathcal{S}$ is a system with exactly~$n+2$ equations in the~$nk+n+1$ variables $z_0, \ldots, z_{nk-1}, x_1, \ldots, x_n, u$ (here $t$ and $\epsilon$ are parameters). We wish to construct a zero-dimensional ideal, so we introduce the \emph{duplicated system} $\Sdup := (\mathcal{S}_1, \ldots, \mathcal{S}_{nk})$, defined in~$\mathbb{L}(t)[x_1, \ldots, x_{n^2k},z_0,\dots,z_{nk-1}, u_1, \ldots, u_{nk}]$.
This system is built from $nk(n+2)$ equations and~$nk(n+2)$ variables.

The following lemma is proven in~\cite[Lemma~2.10]{BoChNoSa22} as a consequence of Hilbert's Nullstellensatz and \cite[Theorem~16.19]{Eisenbud95}:
\begin{lem} \label{lem:jac_ideal}
    Assume that the Jacobian matrix $\Jac_{\Sdup}$ of~$\Sdup$, considered with respect to the variables 
    $x_1, \ldots, x_{n}, u_1, \ldots, x_{n^2k-n}, \ldots, x_{n^2k}, u_{nk},$
    $z_0, \ldots, z_{nk-1}$, is invertible at the point
    \begin{align*}
        \mathcal{P} = 
        (G_1&(U_1), \ldots, G_n(U_1), U_1, \ldots, G_1(U_{nk}), \ldots, G_n(U_{nk}), U_{nk}, {G}_1(0), \ldots, \partial_u^{k-1}{G}_1(0), \ldots, \\
        &{G}_n(0), \ldots, \partial_u^{k-1}{G}_n(0))
        \in\overline{\LL}[[t^{\frac{1}{\star}}]]^{nk(n+1)}\times\LL[[t]]^{nk}.
    \end{align*}
    Then the saturated ideal $\langle \Sdup \rangle : \det(\Jac_{\Sdup})^\infty$ is zero-dimensional and radical over $\LL(t)$. Moreover, $\mathcal{P}$ lies in the zero set of $\langle \Sdup \rangle : \det(\Jac_{\Sdup})^\infty$.
\end{lem}

Therefore, in order to conclude the algebraicity of $G_i(0), \ldots, \partial_u^{k-1}G_i(0)$ over $\LL(t)$ for all $1\leq i \leq n$, it is enough to justify that $\Jac_{\Sdup}$ is invertible at $\mathcal{P}$. Then, by \cref{lem:jac_ideal}, it will follow that it is possible to apply effective techniques from polynomial elimination theory and find annihilating polynomials for the power series of interest.

The idea for proving that $\det(\Jac_{\Sdup})(\mathcal{P}) \neq 0$, analogous to the proof in \cite{BMJ06}, is to show first that $\Jac_{\Sdup}(\mathcal{P})$ can be rewritten as a block triangular matrix. We will then show that each such block is invertible by carefully analyzing its lowest valuation in~$t$.

If $A \in\LL[t][X_1, \ldots, X_n, u]$,
we shall define its ``$i^\text{th}$ duplicated polynomial'' as $A^{(i)} := A(X_{ni+1}, \ldots, X_{n(i+1)}, u_i)$ . Then the Jacobian matrix $\Jac_{\Sdup}(\mathcal{P})$ has the shape
\begin{align*}
    \Jac_{\Sdup}(\mathcal{P}) = \begin{pmatrix}
    A_1 &        & 0  & B_1\\
        & \ddots &    & \vdots \\
    0   &        & A_{nk} & B_{nk}\\
    \end{pmatrix} \in \overline{\LL}[[t^{\frac{1}{\star}}]]^{nk(n+2) \times nk(n+2)},
\end{align*}
where the matrices $A_i \in  \overline{\LL}[[t^{\frac{1}{\star}}]]^{(n+2) \times (n+1)}$ and $B_i \in  \overline{\LL}[[t^{\frac{1}{\star}}]]^{(n+2) \times nk}$ are given by:
\begin{align*}
 \footnotesize   
 A_i := \begin{pmatrix}
    \partial_{x_1}E_1^{(i)}(U_i) & \dots & \partial_{x_n}E_1^{(i)}(U_i)  & \partial_{u_i}E_1^{(i)}(U_i)  \\
    \vdots & \ddots & \vdots & \vdots \\
    \partial_{x_1}E_n^{(i)}(U_i)  & \dots & \partial_{x_n}E_n^{(i)}(U_i)  & \partial_{u_i}E_n^{(i)}(U_i)  \\
     \partial_{x_1}\Det^{(i)}(U_i)  & \dots & \partial_{x_n}\Det^{(i)}(U_i)  & \partial_{u_i}\Det^{(i)}(U_i)  \\
    \partial_{x_1}P^{(i)}(U_i)  & \dots & \partial_{x_n}P^{(i)}(U_i)  & \partial_{u_i}P^{(i)}(U_i)  
    \end{pmatrix}, B_i :=
    \begin{pmatrix}
    \partial_{z_0}E_1^{(i)}(U_i)  & \dots & \partial_{z_{nk-1}}E_1^{(i)}(U_i) \\
    \vdots & \ddots & \vdots \\
    \partial_{z_0}E_n^{(i)}(U_i) & \dots & \partial_{z_{nk-1}}E_n^{(i)}(U_i) \\
    \partial_{z_0}\Det^{(i)}(U_i)  & \dots & \partial_{z_{nk-1}}\Det^{(i)}(U_i) \\
    \partial_{z_0}P^{(i)}(U_i)  & \dots & \partial_{z_{nk-1}}P^{(i)}(U_i)
    \end{pmatrix}.
\end{align*}\vspace{-0.3cm}

Using $\Det(U_i) = 0$ and (\ref{eqn:En_diff_u}), we see that the first $n \times (n+1)$ submatrix of each $A_i$ has rank at most $n-1$. Hence, after performing operations on the first $n$ rows, we can transform the $n^\text{th}$ row of $A_i$ into the zero vector. It follows that after the suitable transformation and a permutation of rows, $\Jac_{\Sdup}(\mathcal{P})$ can be rewritten as a block triangular matrix. To give the precise form of the determinant of $\Jac_{\Sdup}(\mathcal{P})$, we first define 
\begin{equation} \label{eq:Rdef}
\small
  R := \det    \begin{pmatrix}
   \partial_{x_1}E_1^{(i)}(U_i) & \dots & \partial_{x_{n-1}}E_1^{(i)}(U_i) & y_1 \\
    \vdots & \ddots & \vdots  & \vdots \\
    \partial_{x_1}E_{n}^{(i)}(U_i)  & \dots & \partial_{x_{n-1}}E_{n}^{(i)}(U_i) & y_n 
    \end{pmatrix}\in\mathbb{K}[\{\partial_{x_\ell} E_j^{(i)}(U_i)\}_{1\leq \ell, j\leq n}][y_1, \ldots, y_n].
\end{equation}\vspace{-0.3cm}
Then it follows that $
\det(\Jac_{\Sdup})(\mathcal{P}) = \pm \big(\prod\limits_{i = 1}^{nk}\det(\Jac_{i}(U_i))\big)\cdot\det(\Lambda)$, where
\begin{equation*}
\small
    \Jac_i(u) := 
     \begin{pmatrix}
    \partial_{x_1}E_1^{(i)}(u) & \dots & \partial_{x_n}E_1^{(i)}(u)  & \partial_{u_i}E_1^{(i)}(u)  \\
    \vdots & \ddots & \vdots & \vdots \\
    \partial_{x_1}E_{n-1}^{(i)}(u)  & \dots & \partial_{x_n}E_{n-1}^{(i)}(u)  & \partial_{u_i}E_{n-1}^{(i)}(u)  \\
     \partial_{x_1}\Det^{(i)}(u)  & \dots & \partial_{x_n}\Det^{(i)}(u)  & \partial_{u_i}\Det^{(i)}(u)  \\
    \partial_{x_1}P^{(i)}(u)  & \dots & \partial_{x_n}P^{(i)}(u)  & \partial_{u_i}P^{(i)}(u)  
    \end{pmatrix} \in  \LL[u][[t]]^{(n+1) \times (n+1)}, \text{ \normalsize{and}}
\end{equation*}
\begin{equation} \label{eq:def_Lambda}
\small
   \Lambda := \big( R(\partial_{z_j} E_1^{(i)}(U_i), \ldots, \partial_{z_j} E_n^{(i)}(U_i))\big)_{1 \leq i, j+1 \leq nk} \in \overline{\LL}[[t^{\frac{1}{\star}}]]^{nk \times nk}.
\end{equation}
The proof that this product is non-zero is the content of \cref{lem:det_Jaci} and \cref{lem:Lambda}.
\begin{lem}\label{lem:det_Jaci}
    For each $i=1,\dots,nk$, the determinant of $\Jac_i(U_i)$ is non-zero.
\end{lem}
\begin{proof}[Proof sketch]
To prove that $\det(\Jac_{i}(U_i))\neq 0$ we will show that $\textsf{val}_t(\det(\Jac_{i}(U_i)))<\infty$, where $\val_t$ denotes the valuation in~$t$.
The main idea here is to expand~$\det(\Jac_{i}(U_i))$ with respect to the last column and show that the least valuation comes from the product of~$\partial_{u_i}\Det^{(i)}(U_i)$ by its associated submatrix, which we denote by $\mathcal{M}$. 

Since $\partial_{x_j}\Det^{(i)}(U_i) = O(t^\alpha)$, it is clear that for $1\leq j \leq n$ the minors associated to $\partial_{u_i}E_{j-1}^{(i)}(U_i)$ and $\partial_{u_i}P^{(i)}(U_i)$ are in $O(t^\alpha)$. It remains to show that the product of $\partial_{u_i}\Det^{(i)}(U_i)$ by $\det(\mathcal{M})$ is of valuation in~$t$ strictly lower than~$\alpha$. For $j=1,\dots,n-1$ and $\ell=1,\dots,n$ one computes that $\val_t(\mathcal{M})_{j,\ell} = \partial_{x_j}E_\ell^{(i)}(U_i) = \beta+\frac{m_j}{k}$ if $j \neq \ell$ and $\frac{m_j}{k}$ if $j=\ell$. Moreover, it follows from the definition of~$P$ and expansion along the last row of the matrix which defines~$P$ that the term with lowest $t$-valuation in~$\partial_{x_n}P^{(i)}$ is given by the product of~$\partial_{x_n, u_i}E_n^{(i)}$ by the determinant of the associated sub-matrix of $\partial_{u_i}E_n^{(i)}$. Computing this valuation while using that $\alpha,\beta$ are chosen sufficiently large, we find that $\val_t(\partial_{x_n}P^{(i)}) = (\sum_{i = 1}^{n}{\frac{m_i}{k}}) - \frac{1}{k} = (M-1)/k$. It follows that the only monomial in the determinant of $\mathcal{M}$ that has no dependency on $\beta$ comes from the product of diagonal elements of $\mathcal{M}$.
Using the definition $\alpha=n^2k\cdot (\lfloor 2M/k \rfloor + 1) + nM$ and $\beta = \lfloor 2 {M}/{k} \rfloor$, we conclude that $ \textsf{val}_t(\partial_{u_i}\Det^{(i)}(U_i)\cdot\det(\mathcal{M}))<\alpha$. 
\end{proof}

\begin{lem}\label{lem:Lambda}
    The determinant of $\Lambda$ is non-zero.
\end{lem}
\begin{proof}[Proof sketch]
Proving $\det(\Lambda)\neq 0$ is again done by analyzing the first terms of $\det(\Lambda)$. We prove that mod $t^\alpha$ the determinant factors as a product of $U_i$, the Vandermonde determinant $\prod_{i<j} (U_i - U_j)$, and a non-zero polynomial $H(t)$. The actual computation is somewhat technical, since $\det(\Lambda)$ is defined as the determinant of the $nk \times nk$ matrix $R(\partial_{z_j} E_1^{(i)}(U_i), \ldots, \partial_{z_j} E_n^{(i)}(U_i))_{i,j}$ whose entries are themselves determinants of the $n\times n$ matrices~(\ref{eq:Rdef}). We shall give an exposition of the proof, omitting technical details.

We denote $R_j(u) := R(\partial_{z_j} E_1^{(i)}(u), \ldots, \partial_{z_j} E_n^{(i)}(u))$ and compute $R_j(u_i)\bmod t^\alpha$ for variables $u_1,\ldots,u_{nk}$. Note that the latter is a non-zero polynomial in~$\LL[u_1,\dots,u_{nk}, t]$ which is independent of the polynomials $Q_1,\dots,Q_n$. 
Let $\tilde{\Lambda} = (R_j(u_i))_{1 \leq i,j+1 \leq nk}$ be the matrix $\Lambda$ with the $U_i$ replaced by the variables $u_i$.
With tedious but explicit computations it is possible to show that each element in the $i^{\text{th}}$ row of $\tilde{\Lambda} \bmod t^\alpha$ is a polynomial in $u_i$ of degree $\leq M-1$ and valuation $\geq M-nk$. 
Moreover, all entries of~$\tilde{\Lambda}\bmod t^\alpha$ have degree in~$t$ bounded by~$t^{n(\beta+1)}$. The choice for $\alpha$ and $\beta$ ensures 
that~$\det(\tilde{\Lambda} \bmod t^\alpha) = \det(\tilde{\Lambda}) \bmod t^\alpha$. 

As we have $M\geq nk$, it is possible to factor out $u_i^{M - nk}$ from the $i^{\text{th}}$ row of $\tilde{\Lambda}\bmod t^\alpha$ when computing its determinant. This yields polynomials of degree at most $nk-1$ in $u_i$ on the $i^{\text{th}}$ row. Moreover, it is obvious that if $u_i = u_j$ for some $i\neq j$, the determinant of $\tilde{\Lambda}$ vanishes. Hence, we can also factor out the Vandermonde determinant $\prod_{i<j}{(u_i - u_j)}$. As this latter product is of degree $nk-1$ in $u_i$, we conclude that  
 \begin{equation} \label{eq:detlambdatilde}
     \det(\tilde{\Lambda}) \equiv \prod_{i = 1}^{nk}{u_i^{M - nk}}\cdot 
     \prod_{i<j}{(u_i - u_j)}\cdot H(t) \mod t^\alpha,
 \end{equation}
 for some non-zero polynomial $H\in\mathbb{L}[t]$ whose degree only depends on~$\beta$. Recall that $\tilde{\Lambda}(U_1,\dots,U_{nk}) = \Lambda$, and all $U_i$ are distinct with valuation in $t$ of $1/k$ by \cref{lem:det_sols}. Using this, equation (\ref{eq:detlambdatilde}) and $\alpha > (M-nk)n + n + n^2k(\beta+1)$, we conclude that $\det(\Lambda) \neq 0$. 
 \end{proof}
 
Having now proved that $\det(\Jac_{\Sdup}) \neq 0$ at $\mathcal{P}$, we can apply \cref{lem:jac_ideal} and obtain that the specialized series $G_i(0), \ldots, \partial_u^{k-1}G_i(0)$ are all algebraic over~$\mathbb{K}(t, \epsilon)$. The algebraicity of the complete formal power series $G_1, \ldots, G_n$ over~$\mathbb{K}(t, u, \epsilon)$ then follows again by \cite[Lemma~2.10]{BoChNoSa22} from the invertibility of the Jacobian matrix of $E_1, \ldots, E_n$ considered with respect to the variables $x_1, \ldots, x_n$ (with $t, u, z_0, \ldots, z_{nk-1}$ viewed as parameters). 
The equalities $F_i(t^{\alpha}, u) = G_i(t, u, 0)$ finally imply that $F_1, \ldots, F_n$ are also algebraic over~$\mathbb{K}(t, u)$.

As already mentioned, a strength of the presented method is that it is effective. Recall that \cref{thm:quantitative_estimates} summarizes a bound on the algebraicity degree of all $F_i(t,u)$ and estimates the arithmetic complexity of the algorithm which computes $F_i(t,a)$.  
\begin{proof}[Proof sketch of \cref{thm:quantitative_estimates}]
Using the definition of $\alpha$ and $\beta$ in the proof of~\cref{thm:main_thm}, the result is proven along the same lines as the results in~\cite[Section~$3$]{BoChNoSa22}. The algebraicity bound of the shape ${n^{2n^2k^2}(k+1)^{n^2k^2(n+2)+n}\delta^{n^2k^2(n+2)+n}}/{(nk)!^{nk}}$ is a consequence of Bézout's theorem applied to the saturated ideal defined in \cref{lem:jac_ideal}, while the announced complexity is a consequence of~\cite[Theorem~$2$]{Schost03}. Note that the factor $(nk)!^{nk}$ comes from exploiting the structure of the duplicated system by prescribing an action of the symmetric group $\mathfrak{S}_{nk}$ on it.
\qedhere
\end{proof}

\section{Summary and future work} \label{sec:complexity}
   
We can summarize the strategy presented in \cref{sec:proof_thm} as follows: 
\begin{enumerate}
    \item Set up the deformed system~(\ref{eqn:deformed_system}) and the polynomials $\Det, P \in \LL[X_1,\dots,X_n,t,u]$:
    \begin{align*}
    \small
        \Det := 
        \det \begin{pmatrix}
            \partial_{x_1}E_1 & \dots & \partial_{x_n}E_1\\
            \vdots & \ddots & \vdots\\
            \partial_{x_1}E_n & \dots & \partial_{x_n}E_n\\
        \end{pmatrix} \quad
        \text{ \normalsize{and} } \quad
        P:= 
        \det 
    \begin{pmatrix}
    \partial_{x_1}E_1 & \dots & \partial_{x_{n-1}}E_{1} & \partial_{u}E_1\\
    \vdots & \ddots & \vdots & \vdots \\
    \partial_{x_1}E_n & \dots & \partial_{x_{n-1}}E_{n} & \partial_{u}E_n\\
    \end{pmatrix}.
    \end{align*}
    \item Set up the duplicated polynomial system $\Sdup$, consisting of the duplications of the polynomials $E_i,\Det,P$. It has $nk(n+2)$ variables and equations.
    \item Compute a non-trivial element of the saturated ideal $\langle \Sdup \rangle : \det(\Jac_{\Sdup})^\infty$.
\end{enumerate}

As illustrated in \cref{ex:eq27}, the deformation step is not always needed. In fact, it is clear that for a generic system the equation $\Det(u) = 0$ will have $nk$ distinct non-zero solutions in~$\overline{\K}[[t^{\frac{1}{\star}}]]$. Moreover, generically, a non-trivial element of both $\langle \Sdup \rangle : \det(\Jac_{\Sdup})^\infty$ and
$\langle \Sdup \rangle : (\prod_{i\neq j}{(u_i - u_j)})^\infty$
contains  the sought annihilating polynomial. In general, however, the deformation is important, as the following example shows:

\begin{ex2}
For $s$ randomly chosen in~$\mathbb{Q}$ in~(\ref{eq:exhardpart}), one cannot apply Algorithm~\ref{algo:algo1}
because the ideal~$\langle \mathcal{S}_0, \mathcal{S}_1, m\cdot(u_1 - u_2) -1\rangle$ is not $0$-dimensional, despite the fact $\Det(u) = 0$ has $2$ distinct solutions. However, as predicted by the theory, after the deformation (\ref{eqn:deformed_system}) the system indeed becomes zero-dimensional and can be solved systematically, even though the actual computation becomes quite heavy.
\end{ex2}

Our strategy produces a polynomial system with $nk(n+2)+1$ variables and equations (the additional variable and equation come from the saturation). Since, already for small values of $n,k$, solving such systems are often out of reach, we wish to briefly introduce an approach that has a better algorithmic complexity. The idea is to reduce (by eliminating $F_2, \ldots, F_n$) the initial system to a single functional equation $R = 0$, and then to use Bousquet-Mélou and Jehanne's method~\cite{BMJ06}. This reduces to solving a polynomial system with just $3nk$ variables and equations. In order to make this approach work, there are two necessary conditions: the equation $\partial_{x_1} R=0$ should contain enough (that is $nk$) roots in $\overline{\K}[[t^{\frac{1}{\star}}]]$ and the corresponding ideal should be zero-dimensional. Note that $R$ is not a DDE anymore in general, so these conditions are not guaranteed. The following proposition ensures that our deformation takes care of the first part, and the example right after shows that the second condition can still fail in practice. 

\begin{prop}\label{prop:preservation_Ui}
    Let $\textbf{\textup{\text{(E}}}_{\textbf{\textup{\text{F}}}_\textbf{\textup{\text{1}}}} \textbf{\textup{\text{)}}},\dots,\textbf{\textup{\text{(E}}}_{\textbf{\textup{\text{F}}}_\textbf{{\text{n}}}} \textbf{\textup{\text{)}}}$ be as in \cref{thm:main_thm} and let $E_1, \ldots, E_n$ be the polynomials obtained after deforming $\textbf{\textup{\text{(E}}}_{\textbf{\textup{\text{F}}}_\textbf{\textup{\text{1}}}} \textbf{\textup{\text{)}}},\dots,\textbf{\textup{\text{(E}}}_{\textbf{\textup{\text{F}}}_\textbf{{\text{n}}}} \textbf{\textup{\text{)}}}$ as in (\ref{eq:E_i's}). Let
    $U_1, \ldots, U_{nk}\in\overline{\mathbb{K}(\epsilon)}[[t^{\frac{1}{\star}}]]$ be the distinct non-zero series solutions in~$u$
    of the equation $\Det(u)=0$ and let $R\in(\langle E_1, \ldots, E_n \rangle:\Det^\infty)\cap\mathbb{L}[x_1, z_0, \ldots, z_{nk-1},t, u]$. Then~$U_1, \ldots, U_{nk}$ are
    also solutions of 
    ~$\partial_{x_1}R(u) = 0$.
\end{prop}
\begin{proof}
Since $R \in \langle E_1,\dots,E_n \rangle $, there exist $V_1,\dots,V_n \in \LL[x_1,\dots,x_n, z_0, \ldots, z_{nk-1}, t,u]$ such that $R(U_\ell) = \sum_{i=1}^n E_i(U_\ell) V_i(U_\ell)$ for any $\ell = 1,\dots,nk$. Differentiating with respect to $x_j$ for $j=1,\dots,n$ and using that $E_i(U_\ell) = 0$ and that $R$ does not depend on $x_j$ for $j \geq 2$, we find \begin{equation}\label{eq:R_matrix_V}
 \small   \begin{pmatrix}
    \partial_{x_1} R(U_\ell)\\
    0\\
    \vdots\\
    0
    \end{pmatrix} = 
    \begin{pmatrix}
    \partial_{x_1}E_1(U_\ell) & \dots & \partial_{x_1} E_n(U_\ell) \\
    \vdots & \ddots & \vdots \\
    \partial_{x_n}E_1(U_\ell) & \dots & \partial_{x_n} E_n(U_\ell)
    \end{pmatrix}
    \begin{pmatrix}
    V_1(U_\ell)\\
    \vdots\\
    V_n(U_\ell)
\end{pmatrix}.
\end{equation}
By definition of $U_\ell$, the matrix $(\partial_{x_j }E_i(U_\ell))_{i,j}$ is singular and \cref{lem:det_sols} implies that each of its $(n-1)\times(n-1)$ minors is non-zero. It follows that we can express the first row of the matrix as a linear combination of the other rows, then (\ref{eq:R_matrix_V}) implies that $\partial_{x_1} R(U_\ell) = 0$.
\end{proof}

\begin{ex2}
For $s$ randomly chosen in~$\mathbb{Q}$, reducing to a single equation $R$ (by taking the resultant with respect to $x_2$), we indeed find that $\partial_{x_1} R(u) = 0$ has two distinct roots in $\overline{\K}[[t^{\frac{1}{\star}}]]$. However, the computation of a Gröbner basis reveals that the corresponding ideal has positive dimension.
\end{ex2}

\paragraph{Future work.}

The present work proves constructively and elementarily the algebraicity of solutions of systems of DDEs with one catalytic variable. Practical experiments (which are also based on further algorithmic tools under development) make us believe that our method has good potential for practical unresolved combinatorial examples as well.
Moreover, there are three most natural directions for further work. They will deal with complexity improvements for practical computations and theoretical generalizations:
\begin{enumerate}
    \item Exploit the strategy \emph{hybrid guess-and-prove}, which was used in~\cite[Section~$2.2.2$]{BoChNoSa22} to tackle first order scalar DDEs efficiently, and which turns out to be useful when dealing with huge polynomial systems.
    \item \cref{prop:preservation_Ui} ensures that the deformation~(\ref{eqn:deformed_system}) guarantees $nk$ distinct roots of $\partial_{x_1} R(u) = 0$, however, as demonstrated above, the corresponding ideal might still have positive dimension. Investigate whether it is possible to overcome this issue.
    \item Extend the results to a higher number of ``nested'' catalytic variables, where the algebraicity is still guaranteed by Popescu's theorem, but with no effective version.
\end{enumerate}

\paragraph{Acknowledgments.}
We would like to thank Mohab Safey El Din for interesting discussions and Alin Bostan for his great support and many valuable comments on the manuscript. We are also grateful to Mireille Bousquet-Mélou for exciting and enlightening discussions. We thank the anonymous referees for their helpful comments. 

Both authors were supported by the ANR-19-CE40-0018 \href{https://specfun.inria.fr/chyzak/DeRerumNatura/}{De Rerum Natura} project; the second author was funded by the DOC fellowship P-26101 of the \href{https://www.oeaw.ac.at/en/1/austrian-academy-of-sciences}{ÖAW} and the \href{https://oead.at/en/}{WTZ collaboration project} FR-09/2021. 
\vspace{-0.3cm}

\printbibliography

\end{document}